\documentclass[12pt,letterpaper]{amsart}

\usepackage[english]{babel}
\usepackage[utf8]{inputenc}
\usepackage{amsmath,amssymb,amsthm}

\usepackage{bbm,bm}
\usepackage{ushort}
\usepackage{enumitem}
\usepackage{calligra,mathrsfs} 

\allowdisplaybreaks

\newlist{steps}{enumerate}{1}
\setlist[steps]{label=\underline{Step \arabic*:},ref=\arabic*,wide,nosep}

\usepackage{latexsym,mathtools,braket}
\usepackage{tikz-cd}
\usepackage{thmtools,thm-restate}

	\newcommand{\mymarginnote}[1]{} 

\usepackage[colorlinks,plainpages,backref]{hyperref}
\usepackage{cleveref} 

\numberwithin{equation}{subsection}

\declaretheorem[name=Theorem,
	refname={theorem,theorems},
	Refname={Theorem,Theorems},
	numberwithin=section
	]{theorem}
\declaretheorem[name=Proposition,
	refname={proposition, propositions},
	Refname={Proposition, Propositions},
	sibling=theorem]{proposition}
\declaretheorem[name=Lemma,
	refname={lemma,lemmas},
	Refname={Lemma,Lemmas},
	sibling=theorem]{lemma}
\declaretheorem[name=Corollary,
	refname={corollary,corollaries},
	Refname={Corollary,Corollaries},
	sibling=theorem]{corollary}
\declaretheorem[name=Fact,
	refname={fact,facts},
	Refname={Fact,Facts},
	sibling=theorem]{fact}

\declaretheorem[name=Remark,
	refname={remark,remarks},
	Refname={Remark,Remarks},
	style=remark,
	sibling=theorem]{remark}
\declaretheorem[name=Example,
	refname={example,examples},
	Refname={Example,Examples},
	sibling=theorem,
	style=remark]{example}
\declaretheorem[name=Definition,
	refname={definition,definitions},
	Refname={Definition,Definitions},
	style=remark,
	sibling=theorem]{definition}

\renewcommand{\vec}[1]{\mathbf{#1}}

\newcommand{\Gm}{\mathbb{G}_\mathrm{m}}

\newcommand{\CC}{\mathbb{C}}
\newcommand{\ZZ}{\mathbb{Z}}
\newcommand{\NN}{\mathbb{N}}
\newcommand{\RR}{\mathbb{R}}
\newcommand{\QQ}{\mathbb{Q}}

\DeclareMathOperator{\shHom}{\mathscr{H}\text{\kern -3pt {\calligra\large om}}\,}
\DeclareMathOperator{\shRHom}{R\mathscr{H}\text{\kern -3pt {\calligra\large om}}\,}
\DeclareMathOperator{\Spec}{Spec}

\DeclareMathOperator{\cone}{cone}
\DeclareMathOperator{\FL}{FL}

\newcommand{\Ddual}{\mathbb{D}}

\newcommand{\Bigwedge}{\textstyle\bigwedge\!}
\newcommand{\Lder}{\mathrm{L}}

\DeclareMathOperator{\fSupp}{fSupp}
\DeclareMathOperator{\cofSupp}{cofSupp}
\DeclareMathOperator{\sRes}{sRes}
\DeclareMathOperator{\orbit}{O}

\DeclareMathOperator{\MGM}{MGM}
\newcommand{\closure}[1]{\overline{#1}}
\newcommand{\abs}[1]{|#1|}

\DeclareMathOperator{\MHM}{MHM}
\DeclareMathOperator{\dercat}{D}

\newcommand{\sh}[1]{\mathcal{#1}}

\newcommand{\shD}{\sh{D}}

\newcommand{\shK}{\sh{K}}

\newcommand{\shM}{\sh{M}}
\newcommand{\shN}{\sh{N}}
\newcommand{\shO}{\sh{O}}

\newcommand{\rmb}{\mathrm{b}}
\newcommand{\rmc}{\mathrm{c}}

\newcommand{\rmh}{\mathrm{h}}

\begin{document}
\title{Dualizing, projecting, and restricting GKZ systems}

\author{Avi Steiner}
\address{A.~Steiner\\
  Purdue University\\
  Dept.\ of Mathematics\\
  150 N.\ University St.\\
  West Lafayette, IN 47907\\ USA}
\email{steinea@purdue.edu}
\thanks{Supported by the National Science Foundation under grant DMS-1401392.}
\date{\today}

\begin{abstract}
	Let $A$ be an integer matrix, and assume that its semigroup ring $\CC[\NN A]$ is normal. Fix a face $F$ of the cone of $A$. We show that the projection and restriction of an $A$-hypergeometric system to the coordinate subspace corresponding to $F$ are essentially $F$-hypergeometric; moreover, at most one of them is nonzero.
	
	We also show that, if $A$ is in addition homogeneous, the holonomic dual of an $A$-hypergeometric system is itself $A$-hypergeometric. This extends a result from \cite{Wal07}, proving a conjecture of Nobuki Takayama in the normal homogeneous case.
\end{abstract}

\maketitle


\section{Introduction}

Let $A\in \ZZ^{d\times n}$ be an integer matrix with columns $\vec{a}_1,\ldots,\vec{a}_n$ such that $\ZZ A=\ZZ^d$; we abuse notation and also use $A$ to denote the set of its columns. Assume that $\NN A$ is pointed, i.e.~that $\NN{A}\cap -\NN{A}=0$. Associated to this data, Gel$'$fand, Graev, Kapranov, and Zelevinski\u{\i} defined in \cite{GGZ87,GZK89} 
a family of modules over the sheaf $\shD_{\CC^n}$ of algebraic linear partial differential operators on $\CC^n$ today referred to either as \emph{GKZ-} or \emph{$A$-hypergeometric} systems. These systems are defined as follows: 

The \emph{Euler operators} of $A$ are the operators $E_i \coloneqq a_{i1}x_1\partial_1+\cdots+a_{in}x_n\partial_n$ ($i=1,\ldots,d$), and the \emph{toric ideal} of $A$ is the $\CC[\partial_1,\ldots,\partial_n]$-ideal $I_A \coloneqq \Braket{\partial^{\vec{u}_+}-\partial^{\vec{u}_-}|A\vec{u}=0, \vec{u}\in \ZZ^n}$. The \emph{$A$-hypergeometric system} corresponding to the parameter $\beta\in \CC^d$ is then defined to be
\begin{equation}\label{eq:gkz}
	\shM_A(\beta)\coloneqq \shD_{\CC^n}/\left(\shD_{\CC^n}I_A + \shD_{\CC^n}\{E_1-\beta_1,\ldots,E_d-\beta_d\}\right).
\end{equation}
If the condition that $\ZZ A = \ZZ^d$ is relaxed, $\shM_A(\beta)$ may still be defined as above by first choosing a $\ZZ$-basis of $\ZZ A$; the resulting $\shD_{\CC^n}$-module is independent of this choice.

\subsection{Projection and restriction}
Explicit formulas for restriction (i.e.\ pullback via the $D$-module inverse image) to a coordinate subspace were computed in \cite[Th.~4.4]{CJT03} and \cite[Th.~4.2]{FFCJ11} for certain classes of GKZ systems. These formulas were generalized in \cite[Th.~2.2]{FFW11} under certain hypotheses about the genericity of the parameter $\beta$ and the size of the coordinate subspace. We focus on a different situation, and explicitly compute, when the semigroup ring $\CC[\NN A]$ is normal, the restriction of $\shM_A(\beta)$ to the coordinate subspace $\CC^F$ corresponding to a face $F\preceq A$ (see \eqref{eq:CCF} for the notation $\CC^F$). We also compute the projection (i.e.\ the pushforward via the $D$-module direct image) of $\shM_A(\beta)$ to $\CC^F$. Both computations appear in \Cref{th:pr/i}. Note that, unless $F=A$, the subspace $\CC^F$ does not satisfy the size requirements of \cite[Th.~2.2]{FFW11}, hence there is no nontrivial overlap between this paper and \cite{FFW11}. Whereas an earlier version of this article stated that the projection and restriction were equal, this is not actually the case. What is true is in a sense the opposite: at most one of them can be nonzero (\Cref{cor:one-or-the-other}).

Our approach is to use the notion of mixed and dual mixed Gauss--Manin systems (see \S\ref{subsec:mgm}) introduced in \cite{Ste19}. We first study these in slightly more generality in \S\ref{sec:adi}. In \S\ref{sec:tqe}, we generalize the notion of quasi-equivariant $D$-module (introduced by T.~Reichelt and U.~Walther in \cite{rw19}) to what we are calling twistedly quasi-equivariant $D$-modules (\Cref{def:tqi}). We then follow a similar process to that in \cite{rw19} to relate the restriction and projection of such modules (\Cref{lem:pi2i}) and to show that mixed and dual mixed Gauss--Manin systems are twistedly quasi-equivariant (\Cref{prop:MGMquasi-eq}). These results are combined in \S\ref{sec:p-and-r} first to compute the restriction and projection to $\CC^F$ of dual mixed Gauss--Manin and mixed Gauss--Manin systems, respectively (\Cref{th:pr/i}), and then to do the same for normal $A$-hypergeometric systems (\Cref{th:normal-pr/i}).

\subsection{Duality}
N.~Takayama conjectured that the holonomic dual of an $A$-hypergeometric system is itself a GKZ system (after applying the coordinate transformation $x\mapsto -x$ if $A$ is non-homogeneous, i.e.\ if the columns of $A$ do not all lie in a hyperplane). U.~Walther, in \cite{Wal07}, provided a class of counterexamples to this conjecture. However, each of these counterexamples is rank-jumping (i.e.\ the holonomic rank is higher than expected), and in the same paper, Walther shows that for generic parameters, Takayama's conjecture does indeed hold. In particular, when the semigroup ring $\CC[\NN A]$ is normal, he proves (\cite[Prop.~4.4]{Wal07}) that the set of all parameters $\beta$ for which the holonomic dual of $\shM_A(\beta)$ is not a GKZ system has codimension at least three. We show in \Cref{th:dual} using the notion of mixed and dual mixed Gauss--Manin systems that if $A$ is homogeneous, this set is in fact empty.

\subsection*{Acknowledgements}
\thanks{Support by the National Science Foundation under grant DMS-1401392 
is gratefully acknowledged. We would also like to thank Uli Walther for his support and guidance, and Thomas Reichelt and Christine Berkesch for intriguing discussions. In addition, we would like to thank our referees, who allowed us to catch a significant error in a previous version of this article.}

\section{Notation and conventions}
In \S\ref{subsec:geom}, we define various notations and conventions related to varieties, derived categories, $D$-modules, and mixed Hodge modules. \S\ref{subsec:fsupp} recalls the notions of fiber and cofiber support. \S\ref{subsec:toric} defines various notations related to the semigroup $\NN A$, and in \S\ref{subsec:mgm} we recall and discuss the notions of mixed and dual mixed Gauss--Manin parameters and systems.

\subsection{General geometric conventions/notation}\label{subsec:geom}
Varieties, smooth or otherwise, are not required to be irreducible, are defined over $\CC$, and are always considered with the Zariski topology. The closure of a subset $Z$ of a topological space $X$ is written $\closure{Z}$. If $X$ is a smooth variety, denote by $\shD_X$ its sheaf of algebraic linear partial differential operators. A subset $Z$ of a topological space $X$ is \emph{relatively open} if it is an open subset of its closure.

\subsubsection{Derived categories} The category of mixed Hodge modules on a variety $X$ is denoted $\MHM(X)$. The bounded derived category of $\MHM(X)$ is denoted $\dercat^\rmb\MHM(X)$. If $X$ is smooth, the bounded derived category of $\shD_X$-modules with coherent and holonomic cohomology are denoted by $\dercat^\rmb_\rmc(\shD_X)$ and $\dercat^\rmb_\rmh(\shD_X)$, respectively. If $Z$ is a closed subvariety, a superscript $Z$ in the notation for any of these categories denotes the full subcategory of objects whose cohomology is supported in $Z$.

\subsubsection{$D$-module functors} (cf.\ \cite{htt}) The holonomic duality functor (\cite[Def.~2.6.1]{htt}) is denoted $\Ddual$. Let $f\colon X\to Y$ be a morphism of smooth varieties. We write $f_+$ for the $D$-module direct image, $f_\dagger\coloneqq \Ddual f_+ \Ddual$ for the $D$-module exceptional direct image, $f^+\coloneqq \Lder f^*[\dim X-\dim Y]$ for the $D$-module inverse image, and $f^\dagger\coloneqq \Ddual f^+ \Ddual$ for the $D$-module exceptional direct image. If $X_1$ and $X_2$ are smooth varieties and $\shM_i^\bullet\in \dercat^b(\shD_{X_i})$ ($i=1,2$), the exterior tensor product (see \cite[p38]{htt}) of $\shM_1^\bullet$ and $\shM_2^\bullet$ is 
\[ \shM_1^\bullet\boxtimes \shM_2^\bullet \coloneqq \shD_{X_1\times X_2} \otimes_{p_1^{-1}\shD_{X_1}\otimes_\CC p_2^{-1}\shD_{X_2}} (p_1^{-1}\shM_1^\bullet\otimes_\CC p_2^{-1}\shM_2^\bullet).\]

Note that \cite{htt} denotes the functors $f_+$, $f^+$, $f_\dagger$, and $f^\dagger$ by $\int_f$, $f^\dagger$, $\int_{f!}$, and $f^\bigstar$, respectively. They define the first two on pages 33 and 40, respectively, while they define the second two in Def.~3.2.13 on page 91.

\subsubsection{Fourier--Laplace transform} (cf.\ \cite[pp85-102]{Bry86}) The Fourier--Laplace transform is denoted by $\FL$. By definition, $\FL(\shM^\bullet)$ is the pullback of $\shM^\bullet\in \dercat^b(\shD_{\CC^n})$ by the $\CC$-algebra automorphism of $\shD_{\CC^n}$ taking $x_i\mapsto \partial/\partial x_i$ and $\partial/\partial x_i$ to $-x_i$. The inverse Fourier transform is denoted by $\FL^{-1}$ and is defined similarly.

For a description of $\FL$ in terms of $D$-module direct and inverse image functors, see \cite{DE03}.

\subsubsection{Mixed Hodge modules}
Let $\shM^\bullet$ be a complex of mixed Hodge modules, and let $F$ be a functor of $D$-modules. If the mixed Hodge module structure on $\shM^\bullet$ induces a mixed Hodge module structure on $F(\shM^\bullet)$, we will always take $F(\shM^\bullet)$ to be this induced mixed Hodge module unless otherwise specified.

\subsection{Toric and GKZ conventions/notation}\label{subsec:toric}
The semigroup ring of $A$ is $S_A\coloneqq \CC[\NN A]=\CC[\partial_1,\ldots, \partial_n]/I_A$. The toric variey of $A$ is $X_A\coloneqq \operatorname{Var}(I_A)$, and the torus of $A$ is $T_A\coloneqq \Spec \CC[\ZZ A]$. Given $\beta\in \CC^d$, define the $\shD_{T_A}$-module
\begin{equation}\label{eq:OTA}
	\shO_{T_A}^\beta\coloneqq \shD_{T_A}/\shD_{T_A}\!\set{t_i\partial_{t_i} + \beta_i| i=1,\ldots,d} =  \shO_{T_A} t^{-\beta}.
\end{equation}
Note that $\shO_{T_A}^\beta$ can be defined in a coordinate-free manner (see \cite[eq.~(2.1.9)]{Ste19}). Set
\begin{equation}\label{eq:FLMA}
	\hat\shM_A(\beta)\coloneqq \FL^{-1}(\shM_A(\beta)).
\end{equation}

\begin{definition}\label{def:face}
	A submatrix $F$ of $A$ is called a \emph{face} of $A$, written $F\preceq A$, if $F$ has $d$ rows and $\RR_{\geq 0} F$ is a face of $\RR_{\geq 0} A$. A \emph{facet} of $A$ is a face of rank $d-1$.
\end{definition}

Given a face $F\preceq A$, set
\begin{equation}\label{eq:dAF}
    d_{A/F} \coloneqq d - \operatorname{rank}F = d - \dim \RR_{\geq0} F,
\end{equation}
and
\begin{equation}
    n_{A/F} \coloneqq n - \#(\text{columns of }F) = n - \dim{\CC^F}.
\end{equation}

The torus embedding $t\mapsto (t^{\vec{a}_1},\ldots, t^{\vec{a}_n})$ of $T_A$ into $\CC^n$ defined by $A$ induces an action of $T_A$ on $\CC^n$ which makes $T_A$-equivariant the inclusion $X_A\subseteq \CC^n$. If $F\preceq A$ is a face, the $T_A$-orbit of $X_A$ corresponding to $F$ is 
\begin{equation}\label{eq:orbit}
	\orbit_A(F) \coloneqq T_A\cdot \mathbbm1_F,
\end{equation}
where the $i$th coordinate of $\mathbbm1_F$ is $1$ if $\vec{a}_i\in F$ and $0$ otherwise. Set
\begin{equation}\label{eq:CCF}
	\CC^F \coloneqq \Set{x\in \CC^n| x_i=0\text{ for all }\vec{a}_i\notin F}.
\end{equation}

\begin{definition}\label{def:pisf}
	For a facet $G\preceq \NN{A}$, there is a unique linear form $h_G=h_{G,A}\colon \ZZ^d\to \ZZ$, called the \emph{primitive integral support function} of $G$, satisfying the following conditions:
\begin{enumerate}
	\item $h_G(\ZZ^d)=\ZZ$.
	\item $h_G(\vec{a}_i)\geq0$ for all $i$.
	\item $h_G(\vec{a}_i)=0$ for all $\vec{a}_i\in G$.
\end{enumerate}
\end{definition}

\subsubsection{Euler--Koszul complex} We recall the definition of Euler--Koszul complex $\shK^A_\bullet(S_A; E_A-\beta)$ from \cite{MMW05}:
\[ \shK^A_\bullet(S_A; E_A-\beta) \coloneqq K_\bullet\bigl(\cdot(E_A-\beta); \shD_{\CC^n}/\shD_{\CC^n}I_A\bigr);\]
i.e.\ it is the Koszul complex of left $\shD_{\CC^n}$-modules defined by the (right) action of the sequence $E_A-\beta=E_1-\beta_1,\ldots,E_d-\beta_d$ on the left $\shD_{\CC^n}$-module $\shD_{\CC^n}/\shD_{\CC^n}I_A$. The more general Euler--Koszul complexes defined in \cite{MMW05} will not be needed. The inverse Fourier--Laplace transform of $\shK_\bullet^A(S_A; E_A-\beta)$ is denoted by $\hat\shK^A_\bullet(S_A; E_A-\beta)$. Recall also that the zeroth homology sheaf of the Euler--Koszul complex $\shK_\bullet^A(S_A; E_A-\beta)$ is exactly the $A$-hypergeometric system $\shM_A(\beta)$. Moreover, if $S_A$ is Cohen--Macaulay (in particular, by Hochster's Theorem, if $S_A$ is normal), then $\shK^A_\bullet(S_A; E_A-\beta)$ is actually a resolution of $\shM_A(\beta)$ \cite[Th.~6.6]{MMW05}.

\subsection{Mixed and dual mixed Gauss--Manin systems}\label{subsec:mgm}
Given a $T_A$-stable open neighborhood $U\subseteq \CC^n$ of $T_A$ and a $\beta\in \CC^d$, set
\begin{equation}
	\MGM(U,\beta)\coloneqq \varpi_\dagger\iota_+\shO_{T_A}^\beta \quad\text{and}\quad \MGM^*(U,\beta)\coloneqq \varpi_+\iota_\dagger\shO_{T_A}^\beta,
\end{equation}
where $\iota\colon T_A\hookrightarrow U$ is the torus embedding and $\varpi\colon U\hookrightarrow \CC^n$ is inclusion.

\begin{definition}\label{def:mgm}
	A complex $\shM^\bullet\in \dercat^\rmb_\rmh(\shD_{\CC^n})$ is \emph{mixed Gauss--Manin} (resp.\ \emph{dual mixed Gauss--Manin}) if it is isomorphic to $\MGM(U,\beta)$ (resp.\ $\MGM^*(U,\beta)$ for some $U$ and $\beta$.
\end{definition}

\begin{definition}\label{def:mgm-param}
	A parameter $\beta\in \CC^d$ is \emph{mixed Gauss--Manin} (resp.\ \emph{dual mixed Gauss--Manin}) if $\hat\shK_A(S_A; E_A-\beta)$ is mixed Gauss--Manin (resp.\ dual mixed Gauss--Manin).
\end{definition}

Note that the definitions of mixed and dual mixed Gauss--Manin parameters in \cite[Def.~8.15]{Ste19} is different than that in \Cref{def:mgm-param}. However, the two definitions are equivalent by \cite[Th.~8.17 and 8.19]{Ste19}.

\subsection{Fiber and cofiber support}\label{subsec:fsupp}
We recall from \cite[Def.~3.1]{Ste19} the notions of fiber and cofiber support---refer there for main properties along with additional examples. The \emph{fiber support} of a (bounded) complex $\shM^\bullet$ of $\shO_X$-modules is
\begin{equation}
	\fSupp \shM^\bullet \coloneqq \Set{x\in X | k(x)\otimes_{\shO_{X,x}}^\Lder \shM^\bullet_x \neq 0},
\end{equation}
where $k(x)$ denotes the residue field of the point $x\in X$. If $\shM^\bullet\in \dercat^\rmb_\rmc(\shD_X)$, its \emph{cofiber support} is
\begin{equation}
	\cofSupp \shM^\bullet \coloneqq \fSupp \Ddual \shM^\bullet.
\end{equation}
Note that both the fiber support and cofiber support are independent of the complex representing the object $\shM^\bullet\in \dercat^\rmb_\rmc(\shD_X)$.

\begin{example}
    Let $A=\begin{bmatrix}1&1&1&1\\0&1&2&3\end{bmatrix}$ and $\beta=(-1,1)^\top$. We describe the fiber and cofiber support of $\hat\shM_A(\beta)$ using \cite[Lem.~9.1]{Ste19}. The facets of $A$ are $F_1=[\vec{a}_1]$ and $F_2=[\vec{a}_4]$, and their primitive integral support functions (see \Cref{def:pisf}) are $h_1(x,y)=y$ and $h_2(x,y)=3x-y$, resp. Applying these to the vector $\beta$, we get $h_1(\beta)=1\in \NN$ and $h_2(\beta)=-4\in \ZZ_{<0}$. Therefore, $\orbit_A(F_1)$ is in the cofiber support but not the fiber support, $\orbit_A(F_2)$ is in the fiber support but not in the cofiber support, $\orbit_A(\emptyset)$ is in neither, and $\orbit_A(A)$ is in both. In summary,
    \begin{align*}
        \fSupp \hat\shM_A(\beta) &= \orbit_A(F_2) \cup \orbit_A(A)\\
        \intertext{and}
        \cofSupp \hat\shM_A(\beta) &= \orbit_A(F_1) \cup \orbit_A(A).
    \end{align*}
\end{example}

\section{Alternating direct images}\label{sec:adi}
In this section we discuss a generalization of mixed and dual mixed Gauss--Manin systems which we will refer to by the name ``alternating direct images''. 

In \S\ref{subsec:fixed-U}, we characterize in terms of fiber and cofiber support when a $D$-module or mixed Hodge module is isomorphic to a given alternating direct image. 

In \S\ref{subsec:gen-U}, we use the results of \S\ref{subsec:fixed-U} to characterize, under a certain openness condition, when a $D$-module or mixed Hodge module is isomorphic to \emph{some} alternating direct image.

In \S\ref{subsec:gkz-vs-mgm}, we specialize \Cref{cor:gen-plus-dag,cor:gen-dag-plus} to the GKZ case (\Cref{thm:both}). As a consequence, we obtain \Cref{cor:mgm-mgm*}, which states that for GKZ systems, being dual mixed Gauss--Manin is the same as being mixed Gauss--Manin and not rank-jumping.

\subsection{Characterizing alternating direct images passing through a fixed \texorpdfstring{$U$}{U}}\label{subsec:fixed-U}
Let
	\[ Z \xrightarrow{\;\iota\;} U \xrightarrow{\;\varpi\;} X\]
be inclusions of smooth (locally closed) subvarieties, where $U$ is open in $X$, and set $\varphi\coloneqq \varpi\circ  \iota$. We associate to this situation the \emph{alternating direct image} functors $\varpi_+\iota_\dagger$ and $\varpi_\dagger\iota_+$.

\begin{remark}\label{rmk:shrink}
	Note that if $\shN^\bullet$ is in $\dercat^{\rmb,\closure{Z}}_\rmc(\shD_X)$ or $\dercat^\rmb(\MHM^{\closure{Z}}(X))$, then $\varphi^+\shN^\bullet$ is canonically isomorphic to $\varphi^\dagger\shN^\bullet$. To see this, notice that because $\varpi$ is an open embedding, $\varpi^\dagger=\varpi^+$; now shrink $U$ so that $\iota$ is a closed immersion, then apply Kashiwara's equivalence.
\end{remark}

\begin{lemma}\label{lem:gen-plus-dag}
	Let $\shM^\bullet \in \dercat^\rmb_\rmc(\shD_Z)$ (resp.\ $\shM^\bullet\in \dercat^\rmb(\MHM(Z))$). Then $\varpi_+\iota_\dagger \shM^\bullet$ is the unique object in $\dercat^{\rmb,\closure{Z}}_\rmc(\shD_X)$ (resp.\ in $\dercat^\rmb(\MHM^{\closure{Z}}(X))$) such that
	\begin{enumerate}
		\item\label{lem:gen-plus-dag.1} the restriction to $Z$ is isomorphic to $\shM^\bullet$;
		\item\label{lem:gen-plus-dag.2} the fiber support is contained in $U$; and
		\item\label{lem:gen-plus-dag.3} the cofiber support intersected with $U$ is contained in $Z$.
	\end{enumerate}
\end{lemma}
\begin{proof}
	We first show that $\varpi_+\iota_\dagger\shM^\bullet$ satisfies the required properties. Because both $\iota$ and $\varpi$ are inclusions of (locally closed) subvarieties, $\varpi_+\iota_\dagger\shM^\bullet$ is supported on $\closure{Z}$. Applying $\varphi^+$ to $\varpi_+\iota_\dagger\shM^\bullet$, we get
	\[\varphi^+\varpi_+\iota_\dagger\shM^\bullet = \iota^+\iota_\dagger\shM^\bullet = \iota^\dagger\iota_\dagger\shM^\bullet= \shM^\bullet,\]
	where the second equality follows for the same reason as in \Cref{rmk:shrink}. So the restriction to $Z$ is $\shM^\bullet$. Let $i_x$ denote inclusion of a point $x\in X$. If $x\notin U$, then $i_x^+\varpi_+\iota_\dagger \shM^\bullet$ vanishes by \cite[Lem.~3.3]{Ste19}, so the fiber support is contained in $U$. If $x\in U\setminus Z$, then also by \cite[Lem.~3.3]{Ste19},
	\[ i_x^\dagger\varpi_+\iota_\dagger\shM^\bullet = i_x^\dagger\iota_\dagger\shM^\bullet =0.\]
	So, the cofiber support intersected with $U$ is contained in $Z$.
	
	We now prove uniqueness. Suppose $\shN^\bullet$ also satisfies the properties. Then the equality of $\varphi^+\shN^\bullet$ and $\shM^\bullet$ induces a morphism $f\colon \iota_\dagger\shM^\bullet \to \varpi^+\shN^\bullet$. By property~\ref{lem:gen-plus-dag.3}, $i_x^\dagger f = 0$ for all $x\in U\setminus Z$, while by property~\ref{lem:gen-plus-dag.1}, the restriction $\iota^+ f$ is an equality. Hence, $\cone(f)$ has empty fiber support, and therefore it vanishes by \cite[Cor.~3.6]{Ste19}. Thus, $f$ is an isomorphism. By duality, the same argument applied to the case $Z=U$ and $\shM^\bullet=\varpi^+\shN^\bullet$ gives an isomorphism $\shN^\bullet\to\varpi_+\iota_\dagger\shM^\bullet$.
\end{proof}

\begin{lemma}\label{lem:gen-dag-plus}
	Let $\shM^\bullet \in \dercat^\rmb_\rmc(\shD_Z)$ (resp.\ $\shM^\bullet\in \dercat^\rmb(\MHM(Z))$). Then $\varpi_\dagger\iota_+ \shM^\bullet$ is the unique object in $\dercat^{\rmb,\closure{Z}}_\rmc(\shD_X)$ (resp.\ in $\dercat^\rmb(\MHM^{\closure{Z}}(X))$) such that
	\begin{enumerate}
		\item the restriction to $Z$ equals $\shM^\bullet$; 
		\item the cofiber support is contained in $U$; and
		\item the fiber support intersected with $U$ is contained in $Z$.
	\end{enumerate}
\end{lemma}
\begin{proof}
	This follows from \Cref{lem:gen-plus-dag} by duality.
\end{proof}


\begin{remark}
	Let $\shM^\bullet\in \dercat^\rmb(\MHM(Z))$. \Cref{lem:gen-plus-dag,lem:gen-dag-plus} imply that if there are open neighborhoods $U$ and $U'$ of $Z$ such that $\varpi_+\iota_\dagger\shM^\bullet$ and $\varpi'_\dagger\iota'_+\shM^\bullet$ are isomorphic as $\shD_X$-modules, then they are also isomorphic as mixed Hodge modules.
\end{remark}

Finally, we relate the fiber (resp.\ cofiber) supports of $\varpi_+\iota_\dagger\shM^\bullet$ and $\varphi_\dagger \shM^\bullet$ (resp.\ of $\varpi_\dagger\iota_+\shM^\bullet$ and $\varphi_+\shM^\bullet$). Part \eqref{lem:plus-dag-vs-dag-etc.isoms} of the following \namecref{lem:plus-dag-vs-dag-etc} generalizes \cite[Lem.~8.1]{Ste19}.  Recall that a set is relatively open if it is an open subset of its closure.

\begin{lemma}\label{lem:plus-dag-vs-dag-etc}
    Let $\shM^\bullet \in \dercat^\rmb_\rmc(\shD_Z)$ (resp.\ $\shM^\bullet\in \dercat^\rmb(\MHM(Z))$).
    \begin{enumerate}
        \item\label{lem:plus-dag-vs-dag-etc.isoms} There are natural isomorphisms 
        \begin{equation*}
            \varpi_+\iota_\dagger\shM^\bullet \cong \varpi_+\varpi^+\varphi_\dagger\shM^\bullet
            \quad\text{and}
            \quad
            \varpi_\dagger\iota_+\shM^\bullet \cong \varpi_\dagger\varpi^\dagger\varphi_+\shM^\bullet.
        \end{equation*}
        
        \item\label{lem:plus-dag-vs-dag-etc.supps} If $\fSupp\varphi_\dagger\shM^\bullet$ (resp.\ $\cofSupp\varphi_+\shM^\bullet$) is relatively open, then so is $\fSupp\varpi_+\iota_\dagger\shM^\bullet$ (resp.\ $\cofSupp\varpi_\dagger\iota_+\shM^\bullet$).
    \end{enumerate}
\end{lemma}
\begin{proof}
    \eqref{lem:plus-dag-vs-dag-etc.isoms} We prove the first isomorphism. The second follows via duality.
    
    It suffices to show that $\varpi_+\varpi^+\varphi_\dagger\shM^\bullet$ satisfies the three conditions of \Cref{lem:gen-plus-dag}. Conditions \ref{lem:gen-plus-dag.1} and \ref{lem:gen-plus-dag.2} are straightforward from the definitions. To prove condition \ref{lem:gen-plus-dag.3}, observe that
    \[\varpi^+\varpi_+\varpi^+\varphi_\dagger\shM^\bullet 
        \cong \varpi^+\varphi_\dagger\shM^\bullet
        \cong \varpi^+\varpi_+\iota_\dagger\shM^\bullet
        \cong \iota_\dagger\shM^\bullet.\]
    Now apply \cite[Lem.~3.3]{Ste19}.
    
    \eqref{lem:plus-dag-vs-dag-etc.supps} Use part \eqref{lem:plus-dag-vs-dag-etc.isoms} along with \cite[Lem.~3.3]{Ste19}.
\end{proof}

\subsection{The relatively open (co)fiber support case}\label{subsec:gen-U}

If the fiber support of $\varphi_\dagger\shM^\bullet$ is relatively open, then the same is true of $\varpi_+\iota_\dagger\shM^\bullet$ by \Cref{lem:plus-dag-vs-dag-etc}\eqref{lem:plus-dag-vs-dag-etc.supps}. We may therefore shrink $U$ so that $U\cap \closure Z= \fSupp \varpi_+\iota_\dagger \shM^\bullet$ without changing $\varpi_+\iota_\dagger\shM^\bullet$. Similarly, if the cofiber support of $\varphi_+\shM^\bullet$ is relatively open, then we may shrink $U$ so that $U\cap \closure Z= \cofSupp \varpi_\dagger\iota_+ \shM^\bullet$ without changing $\varpi_\dagger\iota_+\shM^\bullet$. As an immediate consequence, we get the following corollaries of \Cref{lem:gen-plus-dag,lem:gen-dag-plus}:

\begin{corollary}\label{cor:gen-plus-dag}
	Let $\shM^\bullet \in \dercat^\rmb_\rmc(\shD_Z)$ (resp.\ $\shM^\bullet\in \dercat^\rmb(\MHM(Z))$), and assume that the fiber support of $\varphi_\dagger\shM^\bullet$ is relatively open. Let $\shN^\bullet\in \dercat^{\rmb,\closure{Z}}_\rmc(\shD_X)$ (resp.\ in $\dercat^\rmb(\MHM^{\closure{Z}}(X))$). Then there exists an open neighborhood $U\subseteq X$ of $Z$ such that (in the notation of \S\ref{subsec:fixed-U}) $\varpi_+\iota_\dagger\shM^\bullet \cong \shN^\bullet$ if and only if all of the following conditions hold:
	\begin{enumerate}
		\item $\varphi^+\shN^\bullet \cong \shM^\bullet$;
		\item $\fSupp \shN^\bullet \cap \cofSupp \shN^\bullet \subseteq Z$; and
		\item\label{cor:gen-plus-dag.relop} $\fSupp \shN^\bullet$ is relatively open.
	\end{enumerate}
\end{corollary}

\begin{corollary}\label{cor:gen-dag-plus}
	Let $\shM^\bullet \in \dercat^\rmb_\rmc(\shD_Z)$ (resp.\ $\shM^\bullet\in \dercat^\rmb(\MHM(Z))$), and assume that the cofiber support of $\varphi_+\shM^\bullet$ is relatively open. Let $\shN^\bullet\in \dercat^{\rmb,\closure{Z}}_\rmc(\shD_X)$ (resp.\ in $\dercat^\rmb(\MHM^{\closure{Z}}(X))$). Then there exists an open neighborhood $U\subseteq X$ of $Z$ such that (in the notation of \S\ref{subsec:fixed-U}) $\varpi_\dagger\iota_+\shM^\bullet \cong \shN^\bullet$ if and only if all of the following conditions hold:
	\begin{enumerate}
		\item $\varphi^+\shN^\bullet \cong \shM^\bullet$;
		\item $\fSupp \shN^\bullet \cap \cofSupp \shN^\bullet \subseteq Z$; and
		\item $\cofSupp \shN^\bullet$ is relatively open.
	\end{enumerate}
\end{corollary}

\subsection{A different characterization of mixed and dual mixed Gauss--Manin parameters}\label{subsec:gkz-vs-mgm}
Specializing \Cref{cor:gen-plus-dag,cor:gen-dag-plus} to the GKZ case, we get \Cref{thm:both} below. Before stating it, we recall the definition of the set of \emph{$A$-exceptional parameters}. This is the set $\mathcal{E}_A$ of parameters $\beta$ for which the holonomic rank of $\shM_A(\beta)$ is larger than for a generic parameter. Note that $\mathcal{E}_A$ also has a description in terms of local cohomology (see \cite{MMW05}).

\begin{theorem}\label{thm:both} Let $\beta\in \CC^d$.
	\begin{enumerate}
		\item\label{thm:both.dualMGM} $\beta$ is dual mixed Gauss--Manin for $A$ if and only if 
	\[\beta\notin \mathcal{E}_A \quad \text{and} \quad \fSupp\hat{\shM}_A(\beta)\cap \cofSupp\hat{\shM}_A(\beta)=T_A.\]
		\item\label{thm:both.MGM} $\beta$ is mixed Gauss--Manin for $A$ if and only if
	\[\fSupp\hat{\shK}^A_\bullet(S_A; E_A-\beta)\cap \cofSupp\hat{\shK}^A_\bullet(S_A; E_A-\beta)=T_A.\]
	\end{enumerate}
\end{theorem}
\begin{proof}
	(\ref{thm:both.dualMGM}) By \cite[Th.~8.17]{Ste19}, a dual mixed Gauss--Manin parameter is not $A$-exceptional. By \cite[Lemma~8.8]{Ste19}, if $\beta\notin \mathcal{E}_A$, then the fiber support of $\hat{\shM}_A(\beta)$ is relatively open; in particular, as $\varphi_\dagger\shO_{T_A}^\beta$ is isomorphic to $\hat{\shM}_A(\beta')$ for some $\beta'\notin \mathcal{E}_A$ (\cite[Rmk.~8.16]{Ste19}), the fiber support of $\varphi_\dagger\shO_{T_A}^\beta$ is relatively open. Now use \Cref{cor:gen-plus-dag}.
	
	(\ref{thm:both.MGM}) By \cite[Prop.~2.2~(4)]{Sai01}, the orbit-cone correspondence, and \cite[Th.~7.4]{Ste19}, the cofiber support of $\hat{\shK}^A_\bullet(S_A; E_A-\beta)$ is relatively open for all $\beta$. In particular, as $\varphi_+\shO_{T_A}^\beta$ is isomorphic to $\hat{\shK}^A_\bullet(S_A;E_A-\beta')$ for some $\beta'$ (\cite[Cor.~3.7]{SW09}), the cofiber support of $\varphi_+\shO_{T_A}^\beta$ is relatively open. Now use \Cref{cor:gen-dag-plus}.
\end{proof}

\begin{corollary}\label{cor:mgm-mgm*} A parameter is dual mixed Gauss--Manin for $A$ if and only if it is mixed Gauss--Manin for $A$ and not $A$-exceptional.
\end{corollary}

\section{Twisted quasi-equivariance}\label{sec:tqe}
Reichelt and Walther introduced in \cite[Def.~3.2]{rw19} the notion of a quasi-equivariant $\shD_E$ module. For the purposes of this paper, we need to generalize this notion slightly (\Cref{def:tqi}) to incorporate a ``twist'' by a rank one integrable connection on $\CC^*$ \`a la \cite{Hot98}. In \Cref{lem:pi2i}, this generalization is used to relate certain projections and restrictions of twistedly equivariant $D$-modules. \Cref{prop:MGMquasi-eq} shows that, when properly interpreted, every mixed and dual mixed Gauss--Manin module is twistedly equivariant. Note that \Cref{lem:pi2i,prop:MGMquasi-eq} are generalizations of \cite[Lem.~3.3 and 3.4]{rw19}.
\bigskip

We begin by recalling the notion of a fibered $\CC^*$-action on a trivial vector bundle. Let $\pi\colon E\to X$ be a trivial vector bundle on a smooth affine variety $X$, and denote by
\[ i\colon X\hookrightarrow E \]
the zero section. Set
\[ E^*\coloneqq E\setminus i(X).\]

\begin{definition}[{\cite[Def.~3.1]{rw19}}]\label{def:fib}
	A $\CC^*$ action $\mu\colon \CC^*\times E \to E$ is \emph{fibered} if
	\begin{enumerate}
		\item $\mu$ preserves fibers;
		\item $\mu$ extends under the inclusion $\CC^*\hookrightarrow \CC$ to a morphism (also denoted $\mu$) $\CC\times E \to E$;
		\item $0\in \CC$ multiplies into the zero section, i.e.\ $\mu\colon \{0\}\times E \to i(X)$; and
		\item $\CC$ fixes the zero section.
	\end{enumerate}
\end{definition}

\begin{definition}\label{def:tqi}
	Let $\mu\colon \CC^*\times E \to E$ be a fibered action on $E$, let $\mu'$ be the restriction of this action to $E^*$, and let $\lambda\in \CC$. A complex $\shM^\bullet\in \dercat^\rmb_\rmh(\shD_E)$ is \emph{$\lambda$-twistedly $\CC^*$-quasi-equivariant} if
	\begin{equation}\label{eq:quasi-eq}
		\mu'^*\shM^\bullet_{|E^*}\cong \shO_{\CC^*}^\lambda \boxtimes \shM^\bullet_{|E^*}.
	\end{equation}
	A complex $\shM^\bullet$ is \emph{twistedly $\CC^*$-quasi-equivariant} if it is $\lambda$-twistedly $\CC^*$-quasi-equivariant for some $\lambda$.
\end{definition}

\begin{remark}\label{rmk:quasi-eq}
	Note that because $\mu'$ is smooth of relative dimension $1$, \eqref{eq:quasi-eq} is equivalent to
	\begin{equation}
		\mu'^+\shM^\bullet_{|E^*} \cong \shO^\lambda_{\CC^*}[1] \boxtimes \shM^\bullet_{|E^*}
	\end{equation}
	and also to
	\begin{equation}
		\mu'^\dagger\shM^\bullet_{|E^*} \cong \shO^\lambda_{\CC^*}[-1] \boxtimes \shM^\bullet_{|E^*}.
	\end{equation}
\end{remark}

The following lemma is proved in exactly the same way as is \cite[Lem.~3.3]{rw19}. The only change to the proof is that ``$\shO_{\Gm}$'' must be replaced throughout with ``$\shO_{\CC^*}^\lambda$''. No issues occur with doing so, and no issues occur with the passage to the derived category as opposed to modules.

\begin{lemma}\label{lem:pi2i}
	If $\shM^\bullet\in\dercat^\rmb_\rmh(\shD_E)$ is $\lambda$-twistedly $\CC^*$-quasi-equivariant, then $\pi_+\shM^\bullet \cong i^\dagger\shM^\bullet$ and $\pi_\dagger\shM^\bullet \cong i^+\shM^\bullet$.
\end{lemma}

We now generalize \cite[Lem.~3.4]{rw19}. The basic idea of the proof is the same. However, sufficiently many technical details need to be modified that we feel it necessary to provide the proof in full.

\begin{proposition}\label{prop:MGMquasi-eq}
	Let $F\preceq A$ be a face, and view $\CC^n$ as a vector bundle over $\CC^F$ via the coordinate projection $\pi\colon\CC^n\to \CC^F$. Let $\beta\in \CC^d$. Then there exists a fibered $\CC^*$-action on $\CC^n$ such that for all $T_A$-stable open neighborhoods $U\subseteq \CC^n$ of $T_A$, both $\MGM(U, \beta)$ and $\MGM^*(U, \beta)$ are twistedly quasi-equivariant.
\end{proposition}
\begin{proof}
	Write $E$ for $\CC^n$ viewed as vector bundle over $\CC^F$. Since $\NN A$ is pointed and $F$ is a face, there exists a $\vec{u}\in \ZZ^d$ such that  $\braket{\vec{a}_i,\vec{u}}=0$ for $\vec{a}_i\in F$ and $\braket{\vec{a}_i, \vec{u}}>0$ for $\vec{a}_i\notin F$. We show that the monomial action $\mu\colon \CC^*\times E\to E$ induced by $\vec{v}\coloneqq A^\top \vec{u}$, i.e.\ $t\cdot (x_1,\ldots, x_n) = (t^{v_1}x_1,\ldots, t^{v_n}x_n)$, satisfies the requirements of the proposition.	
	\begin{steps}
		\item\label{MGMstep.fib} $\mu$ is a fibered action.
		
		\emph{Proof of Step~\ref*{MGMstep.fib}.} Condition (1) of \Cref{def:fib} holds because $v_i=0$ for all $\vec{a}_i\in F$. Because in addition $v_i>0$ for all $\vec{a}_i\notin F$, the action extends to $\CC$; so, condition (2) holds. Conditions (3) and (4) follow immediately from the definition of this extension. This finishes the proof of Step~\ref*{MGMstep.fib}.
	
		\item\label{MGMstep.isom} $\tilde{\mu}^*\shO_{T_A}^\beta\cong \shO_{\CC^*}^{\braket{\vec{u},\beta}}\boxtimes \shO_{T_A}^\beta$, where $\tilde{\mu}$ denotes the monomial action on $T_A$ induced by $\vec{u}$.
		
		\emph{Proof of Step~\ref*{MGMstep.isom}.} Let $f\colon \tilde{\mu}^*\shO_{T_A}^\beta\to \shO_{\CC^*}^{\braket{\vec{u},\beta}}\boxtimes \shO_{T_A}^\beta$ be the $\shO_{\CC^*\times T_A}$-module isomorphism taking the generator $1\otimes t^{-\beta}$ to the generator $s^{-\braket{\vec{u},\beta}}\otimes t^{-\beta}$, where $s$ denotes the coordinate on $\CC^*$. The action of $1\otimes t_i\partial_{t_i}$ on both generators is multiplication by $-\beta_i$, while the action of $s\partial_s$ on both generators is multiplication by $-\braket{\vec{u},\beta}$. Therefore, $f$ is an isomorphism of $\shD_{\CC^*\times T_A}$-modules. This finishes the proof of Step~\ref*{MGMstep.isom}.

		\item\label{MGMstep.tqi} Both $\MGM(U, \beta)$ and $\MGM^*(U, \beta)$ are $\braket{\vec{u},\beta}$-twistedly quasi-equivariant.
		
		\emph{Proof of Step~\ref*{MGMstep.tqi}.} Since the two statements are equivalent via duality, we only prove the first. Consider the following commutative diagram:
		\begin{equation}\label{eq:2squares}
			\begin{tikzcd}
				\CC^*\times T_A \arrow[r, "\mathrm{id}\times\iota'"] \arrow[d, "\tilde{\mu}"]
					&\CC^*\times (U\cap E^*) \arrow[r, "\mathrm{id}\times\varpi'"] \arrow[d, "\mu''"]
					&\CC^*\times E^* \arrow[d, "\mu'"]\\
				T_A \arrow[r, "\iota'"]	&U\cap E^* \arrow[r, "\varpi'"]	&E^*
			\end{tikzcd}
		\end{equation}
		Here, $\iota'$ is the torus embedding, $\varpi'$ is inclusion, $\mu'$ is the restriction of $\mu$ to $E^*$, and $\mu''$ is the restriction of $\mu$ to $U\cap E^*$. By construction, the action $\mu$ factors through the action of $T_A$. So, because $U$ is $T_A$-stable, it is also $\CC^*$-stable, and therefore both squares in \eqref{eq:2squares} are Cartesian. Then
		\begin{align*}
			\mu'^\dagger\MGM(U,\beta)_{|E^*} 
			&\cong \mu'^\dagger\varpi'_\dagger\iota'_+\shO_{T_A}^\beta\\
			&\cong (\mathrm{id}\times \varpi')_\dagger \mu''^\dagger\iota'_+\shO_{T_A}^\beta\\
			&\cong (\mathrm{id}\times \varpi')_\dagger (\mathrm{id}\times \iota')_+ \tilde{\mu}^\dagger\shO_{T_A}^\beta\\
			&\cong (\mathrm{id}\times \varpi')_\dagger (\mathrm{id}\times \iota')_+ (\shO_{\CC^*}^{\braket{\vec{u},\beta}}[-1]\boxtimes \shO_{T_A}^\beta)\\
			&\cong \shO_{\CC^*}^{\braket{\vec{u},\beta}}[-1]\boxtimes \varpi'_\dagger\iota'_+\shO_{T_A}^\beta\\
			&\cong \shO_{\CC^*}^{\braket{\vec{u},\beta}}[-1]\boxtimes \MGM(U,\beta)_{|E^*},
		\end{align*}
		where the second isomorphism is by base change, the third is by base change together with the fact that $\mu''$ and $\tilde{\mu}$ are smooth of the same relative dimension, and the fourth is by Step 2 and the smoothness of $\tilde{\mu}$. Now use \Cref{rmk:quasi-eq}. This finishes the proof of Step~\ref*{MGMstep.tqi} and thereby the \namecref{prop:MGMquasi-eq}.\qedhere
	\end{steps}
\end{proof}

\section{Projections and restrictions}\label{sec:p-and-r}
In \S\ref{subsec:r-and-p}, we use the framework of a $\CC^*$-fibered vector bundle to show that the projection and restriction of alternating direct images are also alternating direct images. We apply this in \S\ref{subsec:r-and-p-gkz} to mixed and dual mixed Gauss--Manin systems.

In \S\ref{subsec:normal}, we specialize these results to the case of normal $S_A$, culminating in \Cref{th:normal-pr/i}, where we compute the restriction and projection of $\shM_A(\beta)$ to the coordinate subspace corresponding to a face of $A$, and \Cref{cor:one-or-the-other}, which says that at most one of the restriction and projection can be nonzero.

\subsection{Restricting and projecting twistedly quasi-equivariant alternating direct images}\label{subsec:r-and-p}
Let $X$ be a smooth affine variety, $\pi\colon E\to X$ a $\CC^*$-fibered vector bundle, and as before, denote by $i\colon X\hookrightarrow E$ the zero section. 
%
%
Consider the following diagrams:
\[ Z \xrightarrow{\quad\iota\quad} U \xrightarrow{\quad \varpi\quad} E \quad\text{and}\quad 
i^{-1}(U)\cap \pi(Z) \xrightarrow{\quad\iota'\quad} i^{-1}(U) \xrightarrow{\quad \varpi'\quad} X.\]
Here, $Z$ is smooth and locally closed in $E$, $U$ is an open subset of $E$ containing $Z$, and the morphisms are inclusion. (Note that the role of $X$ has changed from what it was in \Cref{sec:adi}). Set $\varphi\coloneqq \varpi\circ \iota$ and $\varphi'\coloneqq \varpi'\circ\iota'$. 

\begin{proposition}\label{prop:tqe-adi}
	Let $\shM^\bullet\in\dercat^\rmb_\rmh(\shD_Z)$. Assume that $U\supseteq \pi^{-1}(i^{-1}(U))$ and $\pi(Z)$ is locally closed.
	\begin{enumerate}
		\item\label{prop:tqe-adi.plus-dag} If $\shN^\bullet\coloneqq \varpi_+\iota_\dagger\shM^\bullet$ is twistedly $\CC^*$-quasi-equivariant, then
			\[ i^+\shN^\bullet \cong \varpi'_+\iota'_\dagger (i\circ\varphi')^+\shN^\bullet.\]
		\item\label{prop:tqe-adi.dag-plus} If $\shN^\bullet\coloneqq \varpi_\dagger\iota_+\shM^\bullet$ is twistedly $\CC^*$-quasi-equivariant, then
			\[ \pi_+\shN^\bullet \cong \varpi'_\dagger\iota'_+ (i\circ\varphi')^\dagger\shN^\bullet.\]
	\end{enumerate}	 
\end{proposition}
\begin{proof}
	(\ref{prop:tqe-adi.plus-dag}) By \Cref{lem:gen-plus-dag}, the fiber support of $i^+\shN^\bullet$ is contained in $i^{-1}(U)$. Suppose $x\in i^{-1}(U)\cap \cofSupp i^+\shN^\bullet$. Then by \Cref{lem:pi2i} and the base change formula, $(\pi|_{E_x})_\dagger i_{E_x}^\dagger \shN^\bullet\neq 0$, where $E_x\coloneqq \pi^{-1}(x)$ is the fiber of $E$ over $x$, and $i_{E_x}\colon E_x\hookrightarrow E$ is inclusion. So, $i_{E_x}^\dagger \shN^\bullet\neq 0$, and therefore $E_x \cap \cofSupp \shN^\bullet\neq \emptyset$. On the other hand, $x\in i^{-1}(U)$, so because $U\supseteq \pi^{-1}(i^{-1}(U))$, we have that $E_x\subseteq U$. Hence, $E_x\cap \cofSupp \shN^\bullet$ is a nonempty subset of $Z$  by \Cref{lem:gen-plus-dag}, and therefore $\pi(x)\in \pi(Z)\cap i^{-1}(U)$. Thus, 
	\[ i^+\shN^\bullet \cong \varpi'_+\iota'_\dagger \varphi'^+i^+\shN^\bullet \cong \varpi'_+\iota'_\dagger(i\circ\varphi')^+\shN^\bullet.\]	
	
	(\ref{prop:tqe-adi.dag-plus}) This follows from 
	(\ref{prop:tqe-adi.plus-dag}) by duality together with \Cref{lem:pi2i}.
\end{proof}

It may appear at first that the assumption that $U\supseteq \pi^{-1}(i^{-1}(U))$ in \Cref{prop:tqe-adi} is too restrictive to apply in the situation of \Cref{prop:MGMquasi-eq}. However, as we will see in \Cref{lem:enlargeU}, $U$ can always be enlarged to satisfy this assumption without changing $\MGM(U,\beta)$ or $\MGM^*(U,\beta)$.

\subsection{Restricting and projecting GKZ systems}\label{subsec:r-and-p-gkz}
Before stating \Cref{th:pr/i}, we recall the below facts about mixed and dual mixed Gauss--Manin systems. Also recall from \eqref{eq:orbit} that $\orbit_A(F)$ is the $T_A$-orbit of the toric variety $X_A$ which corresponds to $F$, and from \eqref{eq:dAF} that $d_{A/F}=d - \operatorname{rank}F$.

Here and in the rest of this article, we follow that convention that $\Bigwedge\CC^k$ lives in cohomological degrees $-k$ through $0$.

\begin{fact}\label{fact:mgm*mgm} Let $\beta\in \CC^d$, and let $U\subseteq \CC^n$ be a $T_A$-stable open neighborhood of $T_A$. Write $i_{\orbit_A(F)}$ for the inclusion $\orbit_A(F)\hookrightarrow \CC^n$.
	\begin{enumerate}
		\item\label{fact:mgmcofSupp} If $\orbit_A(F)\subseteq \cofSupp\MGM(U, \beta)$, then 
		\[ i_{\orbit_A(F)}^\dagger \MGM(U, \beta) \cong \bigoplus_{\lambda + \ZZ F} \shO_{T_F}^\lambda \otimes_\CC \Bigwedge \CC^{d_{A/F}},\]
		where the direct sum is over those $\lambda+\ZZ F\in \CC F/ \ZZ F$ for which $\beta-\lambda \in \ZZ^d$. This follows from \cite[Lem.~8.14(b), Rem.~8.16, and Eq.~(8.3.3)]{Ste19}.
		\item\label{fact:mgm*fsupp} If $\orbit_A(F)\subseteq \fSupp\MGM^*(U, \beta)$, then 
		\[ i_{\orbit_A(F)}^+ \MGM^*(U, \beta) \cong \bigoplus_{\lambda + \ZZ F} \shO_{T_F}^\lambda \otimes_\CC \Bigwedge \CC^{d_{A/F}}[-d_{A/F}],\]
		where the direct sum is over those $\lambda+\ZZ F\in \CC F/ \ZZ F$ for which $\beta-\lambda \in \ZZ^d$. This follows from \Cref{fact:mgm*mgm}(\ref{fact:mgmcofSupp}) and \cite[Rmk.~8.18]{Ste19}.
	\end{enumerate}
\end{fact}

Let $F\preceq A$ be a face, and let $\pi_F\colon \CC^n \to \CC^F$ and $i_F\colon \CC^F \hookrightarrow \CC^n$ be coordinate projection and inclusion, respectively.

\begin{lemma}\label{lem:enlargeU}
	Let $\beta\in \CC^n$ and $\shM^\bullet\in \dercat^\rmb_\rmc(\shD_{\CC^n})$. Let $U\subseteq \CC^n$ be a $T_A$-stable open neighborhood of $T_A$, and let $U'=U\cup \pi_F^{-1}(i_F^{-1}(U))$. Then
	\[ \MGM^*(U,\beta) \cong \MGM^*(U',\beta) \quad \text{and} \quad \MGM(U,\beta) \cong \MGM(U',\beta).\]
\end{lemma}
\begin{proof}
	It suffices to show that $U'\cap X_A = U\cap X_A$. The containment $U'\cap X_A \supseteq U\cap X_A$ is immediate. For the other containment, let $G$ be a face of $A$ such that $\orbit_A(G)\subseteq U'$, and suppose $\orbit_A(G)\subseteq \pi_F^{-1}(i_F^{-1}(U))$. Then $i_F(\pi_F(\orbit_A(G)))\subseteq U$. But $i_F(\pi_F(\orbit_A(G)))= i_F(\orbit_F(G\cap F))=\orbit_A(G\cap F)$, so $\orbit_A(G\cap F)\subseteq U$. Therefore, because $U$ is open, the orbit-cone correspondence implies that $\orbit_A(G)\subseteq U$. Thus, $U'\cap X_A = U\cap X_A$.
\end{proof}

\begin{theorem}\label{th:pr/i}
	Let $\beta\in \CC^n$, and let $U\subseteq \CC^n$ be a $T_A$-stable open neighborhood of $T_A$.
	\begin{enumerate}
		\item\label{th:pr/i.bad} If $\beta\notin \CC F + \ZZ^d$ or $U\nsupseteq \orbit_A(F)$, then 
		    \[\pi_{F+}\MGM_A(U,\beta) = i_F^+\MGM^*_A(U,\beta) = 0.\]
		\item\label{th:pr/i.good} If $U\supseteq \orbit_A(F)$, then 
			\begin{align*}
				i_F^+\MGM^*_A(U,\beta) 
					&\cong \bigoplus_{\lambda+\ZZ F} \MGM^*_F\bigl(i_F^{-1}(U), \lambda\bigr)\otimes_\CC \Bigwedge\CC^{d_{A/F}}[-d_{A/F}]
				\intertext{and}
				\pi_{F+}\MGM_A(U,\beta) 
					&\cong \bigoplus_{\lambda+\ZZ F} \MGM_F\bigl(i_F^{-1}(U), \lambda\bigr)\otimes_\CC \Bigwedge\CC^{d_{A/F}},
			\end{align*}
    	where the direct sums are over those $\lambda+\ZZ F\in \CC F/ \ZZ F$ for which $\beta-\lambda \in \ZZ^d$. If in addition, $\beta\in \CC F + \ZZ^d$, then neither $i_F^+\MGM^*_A(U,\beta)$ nor $\pi_{F+}\MGM_A(U,\beta)$ vanish.
	\end{enumerate}
\end{theorem}
\begin{proof}
	We only prove the dual MGM case. The MGM case follows by duality together with \Cref{lem:pi2i}. For ease of notation, set $\pi=\pi_F$ and $i=i_F$.
	
	\eqref{th:pr/i.bad} If $\beta\notin \CC F+ \ZZ^d$ or $U\nsupseteq \orbit_A(F)$, then $\orbit_A(F)$ does not intersect the fiber support of $\MGM^*(U,\beta)$---the former by \Cref{fact:mgm*mgm}\eqref{fact:mgm*fsupp} and the latter by \Cref{lem:gen-plus-dag}\eqref{lem:gen-plus-dag.2}. Therefore, no orbit corresponding to a face of $F$ intersects the fiber support of $\MGM^*(U,\beta)$. Hence, $\fSupp i^+\MGM^*(U,\beta) = X_F \cap i^{-1}(\fSupp \MGM^*(U,\beta)) = \emptyset$, and therefore $i^+\MGM^*(U,\beta)=0$.
	

	\eqref{th:pr/i.good} Assume $U\supseteq \orbit_A(F)$. By \Cref{lem:enlargeU}, we may replace $U$ with $U\cup \pi^{-1}(i^{-1}(U))$ (note that this leaves $i^{-1}(U)$ unchanged) to assume that $U\supseteq \pi^{-1}(i^{-1}(U))$. In addition, $\pi(T_A)=T_F$, which is locally closed in $\CC^F$. Therefore, \Cref{prop:tqe-adi}(\ref{prop:tqe-adi.plus-dag}) applies to give 
	\begin{equation*}
	    i^+\MGM_A^*(U,\beta) \cong \varpi'_+\iota'_\dagger(i\circ\varphi_F)^+\MGM_A^*(U,\beta),
	\end{equation*}
	where $\varphi_F\colon T_F\hookrightarrow \CC^F$ is the torus embedding induced by $F$, $\iota'\colon i^{-1}(U)\cap T_F \hookrightarrow i^{-1}(U)$ is the restriction of $\varphi_F$, and $\varpi'$ is the inclusion $i^{-1}(U)\hookrightarrow \CC^F$.
	
	 By assumption, $i^{-1}(U)\cap T_F=T_F$, and therefore $i\circ\varphi'$ is just the inclusion $T_F=\orbit_A(F)\hookrightarrow \CC^n$. Now use \Cref{fact:mgm*mgm} together with the additivity of the $D$-module functors. This proves that $i_F^+\MGM^*_A(U,\beta)$ is isomorphic to the requisite direct sum. That this does not vanish if $\beta\in \CC F + \ZZ^d$ is because in such a case the direct sum is over a nonempty set.
\end{proof}

\subsection{Normal case}\label{subsec:normal}
Throughout this section, $S_A$ is assumed to be normal. \Cref{lem:gamma} is a technical \namecref{lem:gamma} which we will use (both in this \namecref{subsec:normal} and in \S\ref{sec:duality}) to move a parameter $\beta$ within the class of those parameters whose $A$-hypergeometric system is isomorphic to that of $\beta$. \Cref{lem:lambda} will be needed in the proof of \Cref{th:normal-pr/i}. Recall from \Cref{def:pisf} the definition of the primitive integral support functions $h_G$.

\begin{lemma}\label{lem:gamma}
	Let $\beta\in \CC^d$. Then there exists a $\gamma\in \ZZ^d$ such that for all facets $G\preceq A$,
	\begin{enumerate}
		\item $h_G(\gamma)\neq 0$ if $h_G(\beta)\notin \ZZ$;
		\item $h_G(\gamma)>0$ if $h_G(\beta)\in \NN$; and
		\item $h_G(\gamma)<0$ if $h_G(\beta)\in \ZZ_{<0}$.
	\end{enumerate}
\end{lemma}
\begin{proof}
	Consider the system of equations 
	\[\Set{h_G(x)=h_G(\beta)| G\preceq A\text{ is a facet with }h_G(\beta)\in \ZZ}.\] 
	This has a solution in $\CC^d$, namely $\beta$, and therefore has a solution in $\RR^d$. Let $\alpha$ be one such solution. Then $\alpha$ describes a hyperplane
	\[ H_\alpha = \set{f\in (\RR^d)^*| f(\alpha)=0}.\]
	Denote by $H_\alpha^{\geq0}$ the set of $f\in (\RR^d)^*$ such that $f(\alpha)\geq 0$, and similarly for $H_\alpha^{>0}$, $H_\alpha^{\leq0}$, and $H_\alpha^{<0}$.
	
	Let us now consider the sets $P_\alpha = \set{h_G| h_G(\alpha)\geq0}$ and $N_\alpha = \set{h_G| h_G(\alpha)<0}$. By construction, $\RR_{\geq0}P_\alpha\cap\RR_{\geq0}N_\alpha=\{0\}$. Let $Z$ be an affine hyperplane in $(\RR^d)^*$ transverse to the dual cone $(\RR_{\geq0} A)^\vee$, and assume that the intersection $Z\cap (\RR_{\geq0} A)^\vee$ is nonempty. Then $Z\cap \RR_{\geq0}P_\alpha$ and $Z\cap \RR_{\geq0}N_\alpha$ are convex, compact, and disjoint. Hence, there exists a hyperplane $L$ in $Z$ separating $Z\cap \RR_{\geq0}P_\alpha$ and $Z\cap \RR_{\geq0}N_\alpha$. Choose a $\gamma\in \RR^d$ such that $H_\gamma\cap Z=L$ and $H_\gamma^{>0}\supseteq Z\cap \RR_{\geq0}P_\alpha$. Then $H_\gamma^{<0}\supseteq Z\cap \RR_{\geq0}N_\alpha$. Then by convexity, $H_\gamma^{>0}\supseteq \RR_{\geq0}P_\alpha$ and $H_\gamma^{<0}\supseteq \RR_{\geq0}N_\alpha$. In particular, $H_\gamma^{>0} \supseteq P_\alpha$ and $H_\gamma^{<0}\supseteq N_\alpha$. Because $\QQ^d$ is dense in $\RR^d$, we may modify $\gamma$ so that it is in $\QQ^d$. Clearing denominators, we may take $\gamma$ to be in $\ZZ^d$.
\end{proof}

Note that because we are in the normal case, we may define
\begin{equation}\label{eq:sRes}
	\sRes(A) = \CC^d\setminus\Set{\beta\in \CC^d | h_G(\beta)\geq 0 \text{ whenever } h_G(\beta)\in \ZZ}.
\end{equation}
We will take this as the definition of $\sRes(A)$ since we are only dealing with normal $A$. However, \eqref{eq:sRes} follows from the general definition given in \cite{SW09} by applying \cite[Th.~9.3 and Lem.~9.1]{Ste19} along with \cite[Cor.~3.8]{SW09}

\begin{lemma}\label{lem:lambda}
	Let $\beta\in \CC F + \ZZ^d$, and let $F\preceq A$ be a face. Then there exists a $\lambda\in \CC F \cap (\beta + \ZZ^d)$ such that for all facets $F'$ of $F$,
	\begin{enumerate}
		\item\label{lem:lambda.+} $h_{F'}(\lambda)\in \NN$ implies that $h_G(\beta)\in \NN$ for all facets $G$ of $A$ with $G\cap F = F'$; and
		\item\label{lem:lambda.-} $h_{F'}(\lambda)\in \ZZ_{<0}$ implies that $h_G(\beta)\in \ZZ_{<0}$ for all facets $G$ of $A$ with $G\cap F = F'$.
	\end{enumerate}
\end{lemma}
\begin{proof}
	\begin{steps}
		\item\label{lambdaStep.ZZ} The \namecref{lem:lambda} holds for $\beta\in \ZZ^d$.
	
		\emph{Proof of Step~\ref*{lambdaStep.ZZ}.} By induction on the rank of $F$, we may assume that $F$ is a facet of $A$. Let $F_1,\ldots, F_\ell$ be the facets of $F$. For each $i$, let $G_i$ be the facet of $A$ whose intersection with $F$ is $F_i$. For each $I\subseteq \{1,\ldots, \ell\}$, consider the sets
		\begin{align*}
			X_I	&\coloneqq \Set{x\in\RR F | h_{F_i}(x)\geq 0 \text{ for all }i\in I}\\
			Y_I	&\coloneqq \Set{x\in\RR^d | h_{G_i}(x)\geq 0 \text{ for all }i\in I}.
		\end{align*}
		When $X_I$ is nonempty, neither is $Y_I$, and $X_I$ and $Y_I$ are chambers of the arrangments $\{\RR F_1,\ldots, \RR F_\ell\}$ and $\{\RR G_1,\ldots,\RR G_\ell\}$, respectively. But these two arrangements are combinatorially equivalent by construction, so they have the same number of chambers. Hence, $X_I$ is nonempty if and only if $Y_I$ is nonempty. Since both arrangements are central, $X_I\cap \ZZ F$ is nonempty if and only if $Y_I\cap \ZZ^d$ is nonempty. Therefore, if $\beta\in Y_I$, then any $\lambda\in X_I\cap \ZZ F$ has the required properties. This finishes the proof of Step~\ref*{lambdaStep.ZZ}.
		
		\item\label{lambdaStep.gen} The \namecref{lem:lambda} holds for general $\beta$.
		
		\emph{Proof of Step~\ref*{lambdaStep.gen}.} Apply \Cref{lem:gamma} to $\beta$ to get a $\gamma\in \ZZ^d$. Apply Step~\ref{lambdaStep.ZZ} to $\gamma$ to get an $\alpha\in \ZZ F$. Let $\lambda_0\in \CC F \cap (\beta + \ZZ^d)\setminus \sRes(A)$. By adding sufficiently many copies of $\sum_{\vec{a}_i\in F} \vec{a}_i$ to $\lambda_0$, we may assume that 
		\begin{equation}\label{eq:abs}
			h_{F'}(\lambda_0) \geq \abs{h_{F'}(\alpha)}
		\end{equation}
		for all facets $F'$ of $F$ with $h_{F'}(\lambda_0)\in \ZZ$. Set $\lambda= \lambda_0 + \alpha$. Let $F'$ be a facet of $F$, and let $G$ be a facet of $A$ with $G\cap F = F'$. 
		
		Suppose $h_{F'}(\lambda)\in \NN$. Then because $h_{F'}(\alpha)\in \ZZ$, $h_{F'}(\lambda_0)$ must be an integer and therefore a non-negative integer. Then by \eqref{eq:abs}, $h_{F'}(\alpha)\geq0$. Hence, $h_G(\gamma)\geq 0$, which by construction of $\gamma$ means that $h_G(\beta)\in \NN$. 
		
		Next, suppose $h_{F'}(\lambda)\in \ZZ_{<0}$. As before, this implies that $h_{F'}(\lambda_0)$ is a non-negative integer. But then $h_{F'}(\alpha)$ must be negative. Hence, $h_G(\gamma)\leq 0$, which by construction of $\gamma$ means that $h_G(\beta)\in \ZZ_{<0}$. This finishes the proof of Step~\ref*{lambdaStep.gen} and thereby the \namecref{lem:lambda}.\qedhere
	\end{steps}
\end{proof}

The following example shows that even if $h_G(\beta)\in \ZZ$ for every facet $G$ of $A$ with $G\cap F = F'$, it is still possible that $h_{F'}(\lambda)\notin \ZZ$.

\begin{example}
	Let
	\[ A = \begin{bmatrix}
		1& 1& 1\\
		0& 1& 2
	\end{bmatrix}\qquad \text{and}\qquad F = \begin{bmatrix}1\\2\end{bmatrix}.\]
	The only facet of $F$ is $\emptyset$, and the only facet of $A$ whose intersection with $F$ is $\emptyset$ is the facet $G=[1,0]^\top$. The primitive integral support functions of these facets are $h_{\emptyset, F}(c,2c)=c$ and $h_{G,A}(a,b)=b$. Then $h_{G, A}(c,2c)=2c$, so $h_{G,A}|_{\CC F} = 2 h_{\emptyset, F}$.
	
	Consider the parameter $\beta=(1/2,1)$. This parameter is already in $\CC F$. Since $h_{\emptyset, F}(\beta)=1/2$ is not in $\ZZ$, the same is true of $h_{\emptyset, F}(\lambda)$ for every $\lambda \in \CC F \cap (\beta + \ZZ^2)$. However, $h_{G, A}(\beta)= 2\in \ZZ$. 
\end{example}

Recall that $n_{A/F}$ is the number of columns of $A$ which are not in $F$; equivalently, $n_{A/F} = n - \dim\CC^F$.

\begin{theorem}\label{th:normal-pr/i}
	Assume $S_A$ is normal, let $F\preceq A$ be a face, and let $\beta\in \CC^d$.
	\begin{enumerate}
	    \item\label{th:normal-pr/i.pr} If $\beta\in \CC F + \ZZ^d$ and $h_G(\beta)\in \ZZ_{<0}$ for every facet $G\succeq F$, then there exists a $\lambda\in \CC F \cap (\beta + \ZZ^d)$ such that 
	    \[\pi_{F+}\shM_A(\beta) \cong \shM_F(\lambda)\otimes_\CC \Bigwedge \CC^{d_{A/F}}[n_{A/F}-d_{A/F}];\]
	    otherwise, $\pi_{F+}\shM_A(\beta) = 0$.
	    
	    \item\label{th:normal-pr/i.i} If $\beta\in \CC F + \ZZ^d$ and $h_G(\beta)\in \ZZ_{<0}$ for every facet $G\succeq F$, then there exists a $\lambda\in \CC F \cap (\beta + \ZZ^d)$ such that
	    \[i_F^+\shM_A(\beta) \cong \shM_F(\lambda)\otimes_\CC \Bigwedge \CC^{d_{A/F}}[-n_{A/F}];\]
	    otherwise, $i_F^+\shM_A(\beta) = 0$.
	\end{enumerate}
\end{theorem}
\begin{proof}
    We prove \eqref{th:normal-pr/i.pr}; statement \eqref{th:normal-pr/i.i} is proved similarly.
    
    Recall that the Fourier--Laplace transform interchanges $\pi_{F+}$ and $i_F^+[n_{A/F}]$. Therefore, \eqref{th:normal-pr/i.pr} is equivalent to the following statement (where we recall from \eqref{eq:FLMA} that $\hat\shM_A(\beta)\coloneqq \FL^{-1}(\shM_A(\beta))$):
    \begin{enumerate}
        \item[($\ast$)]\label{th:normal-pr/i.equiv} If $\beta\in \CC F + \ZZ^d$ and $h_G(\beta)\in \ZZ_{<0}$ for every facet $G\succeq F$, then there exists a $\lambda\in \CC F \cap (\beta + \ZZ^d)$ such that 
	    \[i_F^+\hat\shM_A(\beta) \cong \hat\shM_F(\lambda)\otimes_\CC \Bigwedge \CC^{d_{A/F}}[-d_{A/F}];\]
	    otherwise, $i_F^+\hat\shM_A(\beta) = 0$.
    \end{enumerate}
    We prove this Fourier--Laplace transformed statement.
	
	Choose an open subset $U$ of $\CC^n$ such that $U\cap X_A = \fSupp\hat\shM_A(\beta)$. \cite[Th.~9.3]{Ste19} establishes that
	\begin{equation}\label{eq:mgm-vs-gkz}
		\hat\shM_A(\beta) \cong \MGM_A^*(U,\beta).
	\end{equation}
	If $\beta \notin \CC F + \ZZ^d$ or $h_G(\beta)\in \CC \setminus \ZZ_{<0}$ for some facet $G\succeq F$, then $\orbit_A(F)$ is not contained in $U$ by \cite[Lem.~9.1(c)]{Ste19}. \Cref{th:pr/i}(\ref{th:pr/i.bad}) then applies to give that $i_F^+\hat\shM_A(\beta)=0$.
	
	Suppose $\beta\in \CC F + \ZZ^d$ and $h_G(\beta)\in \ZZ_{<0}$ for every facet $G\succeq F$. By normality, the direct sums in \Cref{th:pr/i}(\ref{th:pr/i.good}) collapse to a single summand, giving
	\begin{equation*}
	    i_F^+\hat\shM_A(\beta)\cong \MGM_F^*\bigl(i_F^{-1}(U),\lambda_0\bigr)\otimes_\CC \Bigwedge\CC^{d_{A/F}}[-d_{A/F}],
	\end{equation*}
	where $\lambda_0\in \CC F \cap (\beta + \ZZ^d)$ is arbitrary. Therefore, taking into account \eqref{eq:mgm-vs-gkz}, it remains to show that there exists a $\lambda\in \CC F\cap (\beta + \ZZ^d)$ such that
	\begin{equation*}\label{eq:mgm-vs-Fgkz}
		\MGM_F^*\bigl(i_F^{-1}(U),\lambda_0\bigr) \cong \hat\shM_F(\lambda).
	\end{equation*}
	
	Choose a $\lambda\in \CC F \cap (\beta + \ZZ^d)$ as in \Cref{lem:lambda}. By \cite[Th.~9.3 together with Lem.~9.1(c)]{Ste19}, we need to show for all facets $F'$ of $F$,
	\[ h_{F'}(\lambda)\in \ZZ_{<0} \text{ if and only if } i_F^{-1}(U)\supseteq \orbit_F(F').\]
	Let $F'$ be a facet of $F$. 
	
	If $h_{F'}(\lambda)\in \ZZ_{<0}$, then $h_G(\beta)\in \ZZ_{<0}$ for every facet $G$ of $A$ containing $F'$ by assumption on $\beta$ and by \Cref{lem:lambda}, and $\beta = \lambda + (\beta-\lambda)\in \CC F' + \ZZ^d$; hence, $i_F^{-1}(U)\supseteq \orbit_F(F')$ by \cite[Lem.~9.1(c)]{Ste19}.
	
	If $h_{F'}(\lambda)\in \NN$, then $h_G(\beta)\in\NN$ for some facet $G$ of $A$ containing $F'$, and therefore $i_F^{-1}(U)\nsupseteq \orbit_F(F')$ by \cite[Lem.~9.1(c)]{Ste19}.
	
	Finally, if $\beta\in \CC F' + \ZZ^d$, then $\lambda =\beta + (\lambda-\beta) \in (\CC F' + \ZZ^d)\cap \CC F = \CC F' + \ZZ F$ (because $F$ is saturated), and therefore $h_{F'}(\lambda)\in \ZZ$. Hence, if $h_{F'}(\lambda)\notin \ZZ$, then $\beta\notin \CC F' + \ZZ^d$. Thus, $i_F^{-1}(U)\nsupseteq \orbit_F(F')$ by \cite[Lem.~9.1(c)]{Ste19}.
\end{proof}

Note that \Cref{th:normal-pr/i} only claims the existence of $\lambda$. A possibly interesting question for the future would be to turn the proofs of \Cref{lem:gamma,lem:lambda} into an algorithm for computing such a $\lambda$.

The following corollary follows immediately from \Cref{th:normal-pr/i}:
\begin{corollary}\label{cor:one-or-the-other}
    Assume $S_A$ is normal, let $F\preceq A$ be a face, and let $\beta\in \CC^d$. Then at least one of $i_F^+\shM_A(\beta)$ and $\pi_{F+}\shM_A(\beta)$ is zero.
\end{corollary}

\section{Duality of normal GKZ systems}\label{sec:duality}
Throughout this section, $S_A$ is assumed to be normal. In \Cref{th:dual}, we assume in addition that $A$ is homogeneous (Recall that $A$ is \emph{homogeneous} if its columns all lie in a hyperplane).

\Cref{lem:beta'} shows that for all parameters $\beta$, there is a parameter $\beta'\in -\beta + \ZZ^d$ such that $\hat\shM_A(\beta')$ has the cofiber support one would expect for the holonomic dual of $\hat\shM_A(\beta)$. \Cref{prop:dual} uses this to prove that this $\hat\shM_A(\beta')$ is indeed the holonomic dual of $\hat\shM_A(\beta)$. The Fourier--Laplace transform of this result, together with a monodromicity argument, gives \Cref{th:dual}.

\begin{lemma}\label{lem:beta'}
	Let $\beta\in \CC^d$. Then there exists a $\beta'\in -\beta+\ZZ^d$ such that 
	\[ \cofSupp\hat{\shM}_A(\beta') = \fSupp \hat{\shM}_A(\beta).\]
	If $\beta$ does not lie on the $\CC$-span of any facet, then $\beta'$ may be taken to be $-\beta$.
\end{lemma}
\begin{proof}
	By \cite[Lem.~9.1(c) and (d)]{Ste19}, it suffices to show that there exists a $\beta'\in -\beta+\ZZ^d$ for all facets $G\preceq A$, 
	\begin{equation*}
		h_G(\beta')\in \NN \quad \text{if and only if}\quad h_G(\beta)\in \ZZ_{<0}.
	\end{equation*}
	Choose $\gamma\in \ZZ^d$ as in \Cref{lem:gamma}. Then $\hat{\shM}_A(\beta)$ and $\hat{\shM}_A(\beta+\gamma)$ have the same fiber support (by \cite[Lem.~9.1(c)]{Ste19}) and are therefore isomorphic by \cite[Th.~9.2]{Ste19}. Moreover, $\beta+\gamma$ does not lie on the $\CC$-span of any facet. Replacing $\beta$ with $\beta+\gamma$, we may assume that $\beta$ itself does not lie on the $\CC$-span of any facet.	
	
	Let $\beta'=-\beta$. Then $h_G(\beta)$ is never zero, so $h_G(\beta')\in \NN$ if and only if $h_G(\beta)\in \ZZ_{<0}$, as hoped.
\end{proof}

\begin{proposition}\label{prop:dual}
	Let $\beta\in \CC^d$. Then there exists a $\beta'\in -\beta+\ZZ^d$ such that $\Ddual\hat{\shM}_A(\beta) \cong \hat{\shM}_A(\beta')$.	If $\beta$ does not lie on the $\CC$-span of any facet, then $\beta'$ may be taken to be $-\beta$.
\end{proposition}
\begin{proof}
	By \cite[Th.~9.2]{Ste19}, there exists an open $U\subseteq\CC^n$ with $U\cap X_A=\fSupp \hat{\shM}_A(\beta)$ such that $\hat{\shM}_A(\beta) \cong \MGM^*(U,\beta)$. Applying the holonomic duality functor gives $\Ddual\hat{\shM}_A(\beta) \cong \MGM(U,-\beta)$. Now use \cite[Th.~9.2]{Ste19} again along with \Cref{lem:beta'}.
\end{proof}

\begin{theorem}\label{th:dual}
	Assume that $A$ is homogeneous. Let $\beta\in \CC^d$. Then there exists a $\beta'\in -\beta+\ZZ^d$ such that $\Ddual \shM_A(\beta) \cong \shM_A(\beta')$.	If $\beta$ does not lie on the $\CC$-span of any facet, then $\beta'$ may be taken to be $-\beta$.
\end{theorem}
\begin{proof}
	By \cite[Lem.~1.13]{Rei14}, the homogeneity condition implies that every $A$-hypergeometric system is monodromic. By \cite[Prop.~6.13]{Bry86} (or rather the restatement of it for $D$-modules which appears in \cite[Th.~1.4]{Rei14}), if $\shM$ is monodromic, then $\Ddual \FL \shM\cong \FL \Ddual \shM$. Now use \Cref{prop:dual}.
\end{proof}

\medskip

\bibliographystyle{amsalpha}
\bibliography{gkz}

\end{document}